\documentclass[12pt]{article}
\usepackage[numbers,sort&compress]{natbib}

\usepackage{color}
\usepackage{threeparttable}

\usepackage[english]{babel}
\usepackage{t1enc}
\usepackage{amsthm,amsmath,amssymb}
\usepackage{mathrsfs}
\usepackage[mathscr]{eucal}
\usepackage[latin1]{inputenc}
\usepackage[english]{babel}
\usepackage{latexsym}
\usepackage{graphicx}
\usepackage{lineno}
\usepackage{booktabs}
\usepackage{multirow}
\newtheorem{theorem}{Theorem}

\newtheorem{proposition}{Proposition}
\newtheorem{lemma}{Lemma}

\newtheorem{conjecture}{Conjecture}

\newtheorem{problem}{Problem}
\newtheorem{definition}{Definition}

\makeatletter

\newcommand{\rmnum}[1]{\romannumeral #1}
\newcommand{\Rmnum}[1]{\expandafter\@slowromancap\romannumeral #1@}
\makeatother
\title{Disproof of a conjecture on the minimum Wiener index of signed trees}
\author{\small   Songlin Guo \quad  Wei Wang\thanks{Corresponding author. Email: wangwei.math@gmail.com} \quad Chuanming Wang \\
	\\
	{\footnotesize School of Mathematics, Physics and Finance, Anhui Polytechnic University, Wuhu 241000, P. R. China}
}
\date{}
\begin{document}
	\maketitle
	\begin{abstract}
		The Wiener index of a connected graph is the sum of  distances between all unordered pairs of vertices. Sam Spiro [The Wiener index of signed graphs, Appl. Math. Comput., 416(2022)126755] recently introduced the  Wiener index for a signed graph and conjectured that the path $P_n$ with alternating signs has the minimum Wiener index among all signed trees with $n$ vertices. By constructing an infinite family of counterexamples, we prove that the conjecture is false  whenever $n$ is at least 30.  \\
		
		
	\end{abstract}
\section{Introduction}
\label{intro}

A \emph{signed graph} is a graph where each edge has a positive or negative sign. We usually write a signed graph as a pair $(G,\sigma)$, where $G$ is the underlying graph and $\sigma\colon\,E(G)\mapsto \{+1,-1\}$ describes the sign of each edge.  For a path $P$ in $(G,\sigma)$, the \emph{length} of $P$ (under the signing $\sigma$) is $\ell_\sigma(P)=|\Sigma_{e\in E(P)}\sigma(e)|$. A path $P$ in $(G,\sigma)$ is called  a $uv$-\emph{path} if it has $u$ and $v$ as its endvertices. For two distinct vertices  $u,v\in V(G)$, the \emph{signed distance} \cite{Spiro2022} of $u,v$ in  $(G,\sigma)$, is 
 $$d_{\sigma}(u,v) =\min\{\ell_\sigma(P)\colon\,P \text{~is a~} uv\text{-path~in~} (G,\sigma)\}.$$ 
 
\begin{definition}[\cite{Spiro2022}]\label{signW}\normalfont
Let $(G,\sigma)$ be a signed graph. The Wiener index of $(G,\sigma)$, denoted by $W_\sigma(G)$,  is
$\sum d_{\sigma}(u,v),$
 where the summation is taken over all unordered pairs $\{u,v\}$ of distinct vertices in $G$.
\end{definition}
Let $(G,+)$ denote a signed graph where each edge is positive. It is easy to see that the Wiener index $W_{+}(G)$ coincides with the classic Wiener index $W(G)$ of the ordinary graph $G$, introduced by Harry Wiener \cite{wiener1947} in 1947. As the oldest topological index of a molecule, Wiener index has many applications in molecular chemistry, see the monograph \cite{trinajsti}.

A tree is a connected graph with no cycles. There are numerous studies of properties of the Wiener indices of trees, see the survey paper \cite{dobrynin2001}. Entringer, Jackson and Snyder~\cite{entringer1976} proved that, among all trees of any fixed order $n$, the path $P_n$ (resp. the star $K_{1,n}$) has the maximum (resp. minimum) Wiener index. Note that for any connected graph $G$ together with any signing $\sigma$, we have
$W_\sigma(G)\le W_+(G)=W(G)$. Consequently, the above result of Entringer et al. indicates that $W_\sigma(T)\le W(P_n)$ for any signed $n$-vertex tree $(T,\sigma)$. 

Let $\sigma$ be a signing of the path $P_n$. We call $\sigma$ (or $(P_n,\sigma)$) \emph{alternating} if any two adjacent edges have opposite signs. We usually use $\alpha$ to denote an alternating signing of a path. The following interesting conjecture was proposed recently by Spiro \cite{Spiro2022}.
\begin{conjecture}[\cite{Spiro2022}]\label{conj}
Among all signed trees of order $n$, the alternating path $(P_n,\alpha)$ has the minimum Wiener index.
\end{conjecture}
In this short note, we disprove Conjecture \ref{conj} by constructing infinite counterexamples.
\begin{theorem}\label{main}
	Conjecture \ref{conj} fails for every $n\ge 30$.
\end{theorem}
The proof of Theorem \ref{main} is given at the end of the next section.
\section{An infinite family of counterexamples}
Let $k\ge 0$ and $a_1,a_2,\ldots,a_k$ be $k$ nonnegative integers. Let $T(a_1,a_2,\ldots,a_k)$ denote a rooted tree with $1+k+\sum_{i=1}^k a_i$ vertices constructing by the following two rules:\\
(\rmnum{1}) The root vertex has $k$ neighbors $u_1,u_2,\ldots,u_k$; such $k$ vertices will be called \emph{branch} vertices. \\
(\rmnum {2}) For each $i\in\{1,2,\ldots,k\}$, the branch vertex $u_i$ has $a_i$ neighbors other than the root vertex;  such $a_i$ neighbors will be called \emph{leaf} vertices.

\begin{definition}\normalfont\label{nice}
	Let $\sigma$ be a signing of a rooted tree $T(a_1,a_2,\ldots,a_k)$. We call $\sigma$ \emph{nice} if it satisfies the following two conditions:\\
(\rmnum{1})  Among $k$ edges incident to the root vertex, the numbers of positive edges and negative edges differ by at most one.\\
(\rmnum{2})  For each branch vertex $u$, all edges connecting $u$ and leaf vertices have the same sign which is opposite to the sign of the edge connecting $u$ and the root vertex.
	\end{definition}
Figure 1 illustrates a nice signing for the rooted tree $T(3,4,4,4,4,4)$, where we use dashed (resp. solid) lines to represent negative (resp. positive) edges.
\begin{figure}[htbp]
	\centering
	\includegraphics[height=4cm]{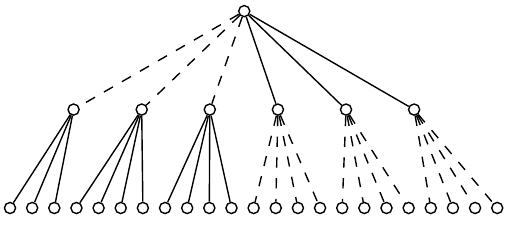}
	\caption{$T(3,4,4,4,4,4)$ with a nice signing.}
	\label{ex16}
\end{figure}
\begin{theorem}\label{basic}
	If $\sigma$ is a nice then $$W_\sigma(T(a_1,a_2,\ldots,a_k))=2\sum_{i=1}^k\binom{a_i}{2} +2\binom{\lfloor\frac{k}{2}\rfloor}2+2\binom{\lceil\frac{k}{2}\rceil}2+k\left(1+\sum_{i=1}^{k} a_i\right).$$
\end{theorem}
\begin{proof}
	Write $T=T(a_1,a_2,\ldots,a_k)$ and let $P$ be any path in $(T,\sigma)$. Clearly, $P$ contains at most four edges. Since $\sigma$ is nice, one easily sees from Definition \ref{nice}(\rmnum{2}) that any path in $(T,\sigma)$ with 4 edges have exactly 2 positive  edges and hence satisfies $\ell_\sigma(P)=0$. Similarly, if $P$ has exactly 2 edges and  $\ell_\sigma(P)>0$ then the two endvertices of $P$ must be either two leaf vertices adjacent to a common branch vertex, or two branch vertices adjacent to the root vertex by two edges sharing the same sign. Note that the numbers of positive edges and negative edges are $\lfloor\frac{k}{2}\rfloor$ and $\lceil\frac{k}{2}\rceil$ (or in reverse order) by Definition \ref{nice}(\rmnum{1}). Thus, the  contribution of such paths to $W_\sigma(T)$ is
	$$2\sum_{i=1}^k\binom{a_i}{2} +2\binom{\lfloor\frac{k}{2}\rfloor}2+2\binom{\lceil\frac{k}{2}\rceil}2.$$
Furthermore, noting that each path $P$ with exactly one or three edges satisfies  $\ell_\sigma(P)=1$ and there exists such a path between branch vertices and the remaining vertices, we see that the contribution of path with one or three edges is exactly 
	$$k\left(1+\sum_{i=1}^{k} a_i\right).$$
Adding the above two expressions completes the proof. 
\end{proof}
\begin{lemma}\label{Wpath}
	Let $\alpha$ be an alternating signing of $P_n$. Then $W_\alpha(P_n)=\lfloor\frac{n}{2}\rfloor\lceil\frac{n}{2}\rceil$.
\end{lemma}
\begin{proof}
	Let $(U,V)$ be the bipartition of $P_n$ as a bipartite graph, where we assume $|U|\le |V|$. Then $|U|=\lfloor\frac{n}{2}\rfloor$ and $|V|=\lceil\frac{n}{2}\rceil$. Let $u,v$ be any two vertices of $P_n$.
	It is easy to see that $d_\alpha(u,v)=0$ if $u$ and $v$ are in the same part, and $d_\alpha(u,v)=1$ otherwise. Thus, $W_\alpha(P_n)=|U||V|=\lfloor\frac{n}{2}\rfloor\lceil\frac{n}{2}\rceil$, as desired.
\end{proof}
Noting that $T(3,4,4,4,4,4)$ has exactly 30 vertices, the following proposition gives a counterexample to Conjecture \ref{conj}.
\begin{proposition}\label{Wpath}
	Let $\alpha$ be an alternating signing of $P_{30}$ and $\sigma$ be a nice signing of $T=T(3,4,4,4,4,4)$. Then $W_\sigma(T)<W_\alpha( P_{30}).$
\end{proposition}
\begin{proof}
	Using Theorem \ref{basic} and Lemma \ref{Wpath}, we find that $W_\sigma(T)=222$ while $W_\alpha( P_{30})=225$. Thus $W_\sigma(T)<W_\alpha( P_{30})$, as desired.
\end{proof}
We shall show that for any $n\ge 30$, there exists a  counterexample to Conjecture \ref{conj}. 
\begin{definition}
$$\mathcal{T}_k=\bigcup_{0\le s\le k}\left\{T(\underbrace{k-1,\ldots,k-1}_{k-s},\underbrace{k,\ldots,k}_s),T(\underbrace{k,\ldots,k}_{k-s},\underbrace{k+1,\ldots,k+1}_s)\right\}.$$
\end{definition}
Note that $\mathcal{T}_k$ contains exactly $2k+1$ rooted trees of consecutive orders from $k^2+1$ to $(k+1)^2$, see Figure \ref{Tk} for the five rooted trees in $\mathcal{T}_2$.
\begin{figure}[htbp]
	\centering
	\includegraphics[height=5cm]{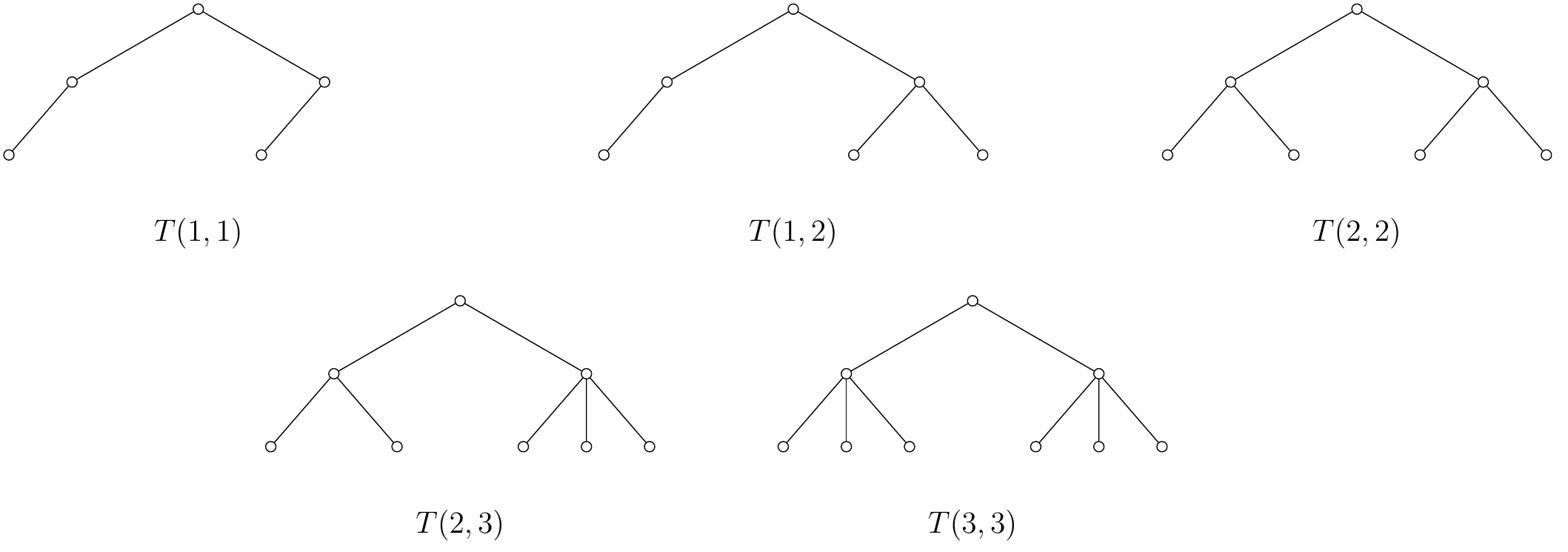}
	\caption{The family $\mathcal{T}_2$.}
	\label{Tk}
\end{figure}

\begin{lemma}\label{atleast8}
Let $k\ge 10$ and $T$ be any rooted tree in $\mathcal{T}_k$. Let $n=|V(T)|$. Then $W_\sigma (T)<W_\alpha(P_n)$ where $\sigma$ is nice while $\alpha$ is alternating.
\end{lemma}
\begin{proof}
	Write $m=k^2+1$ and $M=(k+1)^2$. By Theorem \ref{basic} and Lemma \ref{Wpath}, it is not difficult to see that both  $W_\sigma (T)$ and $W_\alpha(P_n)$ are increasing as a function of $n=|V(T)|$. Thus we are done if we can show that $W_\sigma (T_M)<W_\alpha(P_m)$ where $T_M=T(\underbrace{k+1,\ldots,k+1}_k)$.
	
By Theorem \ref{basic} we have
	 \begin{eqnarray}\label{removebeta}
W_\sigma(T_M)&=&2k\binom{k+1}{2} +2\binom{\lfloor\frac{k}{2}\rfloor}2+2\binom{\lceil\frac{k}{2}\rceil}2+k(1+k(k+1))\label{wtk}\\
	&< &2k\binom{k+1}{2}+2\binom{\frac{k}{2}}{2}+2\binom{\frac{k+1}{2}}{2}+k(1+k(k+1))\nonumber\\
	&=&2k^3+\frac{5}{2}k^2+\frac{1}{2}k-\frac{1}{4}.\nonumber
	\end{eqnarray}
On the other hand, by Lemma \ref{Wpath}, we have 
$$W_\alpha(P_m)=\left\lfloor{\frac{m}{2}}\right\rfloor\left\lceil\frac{m}{2}\right\rceil=\left\lfloor{\frac{k^2+1}{2}}\right\rfloor\left\lceil\frac{k^2+1}{2}\right\rceil>\frac{1}{4}k^4.$$
It follows that
$$\frac{W_\sigma(T_M)}{W_\alpha(P_m)}<\frac{8}{k}+\frac{10}{k^2}+\frac{2}{k^3}-\frac{1}{k^4}<\frac{8}{k}+\frac{10}{k^2}+\frac{2}{k^3}\le \frac{8}{10}+\frac{10}{10^2}+\frac{2}{10^3}<1.$$
Thus $W_\sigma (T_M)<W_\alpha(P_m)$, as desired.	The proof is complete.
\end{proof}
\begin{proof}[Proof of Theorem \ref{main}] Let $\mathcal{T}=\cup_{k=0}^{\infty}\mathcal{T}_k$.  It is clear that $\mathcal{T}$ contains exactly one $n$-vertex (rooted) tree for every positive integer $n$. We use $T_n$ to denote the unique $n$-vertex tree in the family $\mathcal{T}$. Let $\sigma$ be a nice signing of $T_n$ and $\alpha$ be an alternating signing of $P_n$. By Lemma \ref{atleast8}, we see that $W_\sigma(T_n)<W_\alpha(P_n)$ whenever $n\ge 10^2+1$. On the other hand, we know from Proposition  \ref{Wpath} that there does exist a $30$-vertex tree $T$ (with a nice signing $\sigma$) such that $W_\sigma(T)<W_\alpha(P_{30})$. It remains to consider the case that $n\in\{31,32,\ldots,100\}$.
	
	We claim that $W_\sigma(T_n)<W_\alpha(P_n)$ for each $n\in\{31,32,\ldots,100\}$. This can be checked directly using Theorem \ref{basic} and Lemma \ref{Wpath}. Take $n=31$ as an example. As $31\in [5^2+1,(5+1)^2]$, we find that $T_{31}\in \mathcal{T}_5$ and moreover $T_{31}=T(5,5,5,5,5)$. Using Theorem \ref{basic} for $T_{31}$, we obtain that  $W_\sigma(T_{31})=238$. By Lemma \ref{Wpath}, we have $W_\alpha(P_{31})=\lfloor\frac{31}{2}\rfloor\lceil\frac{31}{2}\rceil=240.$ Thus $W_\sigma(T_n)<W_\alpha(P_n)$ for $n=31$. The proof is complete.	
\end{proof}
We remark that the counterexamples constructed in this note also disprove another conjecture of Spiro. For a graph $G$, the \emph{minimal signed Wiener index} of $G$, denoted by $W_*(G)$, is the minimum of $W_\sigma(G)$ for all possible signings  $\sigma$. Spiro \cite{Spiro2022} conjectured that $W_*(T)\ge W_*(P_n)$ for any $n$-vertex tree $T$. Let $n\ge 30$ and $T_n$ be the tree used in the proof of Theorem \ref{main}. Clearly,  $W_*(T_n)\le W_\sigma(T_n)$, where $\sigma$ is a nice signing of $T_n$. On the other hand, it is easy to see that  $W_*(P_n)= W_\alpha(P_n)$. Since $W_\sigma(T_n)<W_\alpha(P_n)$, we obtain $W_*(T_n)<W_*(P_n)$,  disproving this conjecture.  
\section{Asymptotic property} 
It is still unknown which signed trees have the minimum Wiener index among all signed trees of a fixed order $n$. 	We use $(\hat{T_n},\hat{\sigma})$ to denote an $n$-vertex signed tree whose Wiener index is minimum among all signed trees of order $n$. And let $(T_n,\sigma)$ be the $n$-vertex tree in $\cup_{k=0}^{\infty}\mathcal{T}_k$ with a nice signing $\sigma$.   One referee kindly points out that $(T_n,\sigma)$ is optimal up to a constant factor. Precisely, 
$$\limsup\limits_{n\rightarrow\infty} \frac{W_\sigma(T_n)}{W_{\hat\sigma}(\hat{T}_n)}\le C,$$  
for some constant $C$.
 \begin{lemma}\label{wtn}
 $W_\sigma(T_n)=(2+o(1))n^\frac{3}{2}$.
\end{lemma} 
\begin{proof}
	Let $k=\lfloor\sqrt{n-1}\rfloor$, $m=k^2+1$ and $M=(k+1)^2$. Then we have $m\le n\le M$. Note that $T_m=T(\underbrace{k,\ldots,k}_k)$ and $T_M=T(\underbrace{k+1,\ldots,k+1}_k)$. Using Theorem \ref{basic}, we have
	 \begin{equation}
	W_\sigma(T_m)=2k\binom{k}{2} +2\binom{\lfloor\frac{k}{2}\rfloor}{2}+2\binom{\lceil\frac{k}{2}\rceil}{2}+k(1+k^2)=(2+o(1))k^3
	\end{equation}
and 
	 \begin{equation}
W_\sigma(T_M)=2k\binom{k+1}{2} +2\binom{\lfloor\frac{k}{2}\rfloor}{2}+2\binom{\lceil\frac{k}{2}\rceil}{2}+k(1+k(k+1))=(2+o(1))k^3.
\end{equation}
Noting that $k^3\sim n^{\frac{3}{2}}$ and $W_\sigma(T_m)\le W_\sigma(T_n)\le W_\sigma(T_M)$, we have $W_{\sigma}({T}_n)=(2+o(1))n^\frac{3}{2}$ by Squeeze Theorem. 
\end{proof}
The following lower bound  is due to Sam Spiro.
\begin{lemma}\label{optlb}
	$W_{\hat\sigma}(\hat{T}_n)\ge(\sqrt{2}+o(1))n^{\frac{3}{2}}$.
\end{lemma}
\begin{proof}
	Let $U,V$ be the bipartition of $\hat{T}_n$ with $|U|\le |V|$. Label  vertices  in $U$ as $u_1,u_2,\ldots, u_k$, where $k=|U|$. Let $d_i^+$ (resp. $d_i^-$) denote the number of positive (resp. negative) edges incident with $u_i$ for each $i$. It is not too difficult to show that
	\begin{equation}\label{wt}
W_{\hat\sigma}(\hat{T}_n)\ge |U||V|+2\sum_{i=1}^k\left(\binom{d_i^+}{2}+\binom{d_i^-}{2}\right).
	\end{equation}
Indeed, the first term comes from all paths of odd length and the  term  $\binom{d_i^+}{2}+\binom{d_i^-}{2}$ comes from the paths of length 2 between two neighbors of  $u_i$ with the same sign. As the function $\binom{x}{2}=\frac{1}{2}x(x-1)$ is convex, we have
\begin{equation}\label{cv}
\sum_{i=1}^k\left(\binom{d_i^+}{2}+\binom{d_i^-}{2}\right)\ge 2k\binom{\frac{1}{2k}\sum_{i=1}^k(d_i^++d_i^-)}{2},
\end{equation}
by Jensen's Inequality. As $|U|=k$, $|V|=n-k$ and $\sum_{i=1}^k(d_i^++d_i^-)$ equals $n-1$, which is the number of edges in $\hat{T}_n$, we obtain from Eqs. \eqref{wt} and \eqref{cv} that 
	 \begin{eqnarray}\label{WT}
W_{\hat\sigma}(\hat{T}_n)&\ge& k(n-k)+4k\binom{\frac{n-1}{2k}}{2}\nonumber\\
&= &kn+\frac{n^2}{2k}-k^2+\frac{1}{2k}((2k+1)-(2k+2)n)\nonumber\\
&\ge&kn+\frac{n^2}{2k}-k^2-2n.
\end{eqnarray}
Using the basic inequality $a+b\ge 2\sqrt{ab}$ for $a,b>0$, we have
\begin{equation}\label{uvb}
kn+\frac{n^2}{2k}\ge 2\sqrt{{\frac{n^3}{2}}}=\sqrt{2}n^\frac{3}{2}.
\end{equation} Recall that $k\le n/2$. Thus $n-k\ge n/2$. If $k\ge 2\sqrt{2n}$ then from the trivial inequality $W_{\hat\sigma}(\hat{T}_n)\ge k(n-k)$ we obtain 
$$W_{\hat\sigma}(\hat{T}_n)\ge (2\sqrt{2n}) \cdot\frac{n}{2}=\sqrt{2}n^\frac{3}{2}.$$ 
Now assume $k<2\sqrt{2n}$. Then by \eqref{WT} and \eqref{uvb}, we find
$$W_{\hat\sigma}(\hat{T}_n)\ge \sqrt{2}n^\frac{3}{2}-k^2-2n\ge \sqrt{2}n^\frac{3}{2}-10n=(\sqrt{2}+o(1))n^\frac{3}{2}.$$
Thus  we always have $W_{\hat\sigma}(\hat{T}_n)\ge(\sqrt{2}+o(1))n^{\frac{3}{2}}$, as desired.
\end{proof}
The following theorem is a direct consequence of Lemmas \ref{wtn} and \ref{optlb}.
\begin{theorem}
$$\limsup\limits_{n\rightarrow\infty} \frac{W_\sigma(T_n)}{W_{\hat\sigma}(\hat{T}_n)}\le \sqrt{2}.$$ 
\end{theorem}
We end this note by leaving the following problem suggested by one referee.
\begin{problem}
Is it true that $$\lim\limits_{n\rightarrow\infty} \frac{W_\sigma(T_n)}{W_{\hat\sigma}(\hat{T}_n)}=1?$$
\end{problem}
	\section*{Acknowledgments}
	The authors would like to thank the
	anonymous reviewer for her/his instructive suggestions.  In particular, the results in the final section are suggested by the reviewer. We thank Sam Spiro for pointing out a preliminary version of Lemma \ref{optlb}. The second author is supported by the	National Natural Science Foundation of China under the grant number 12001006. 


\begin{thebibliography}{99}
\bibitem{dobrynin2001}A.~A.~Dobrynin, R.~Entringer, I.~Gutman, Wiener index of trees: Theory and applications, Acta Appl. Math. 66(2001) 211-249.
\bibitem{entringer1976} R.~C.~Entringer, D.~E.~Jackson, D.~A.~Snyder, Distance in graphs, Czechoslovak Math. J. 26(1976) 283-296.
\bibitem{Spiro2022} S.~Spiro, The Wiener index of signed graphs, Appl. Math. Comput., 416(2022)126755.

\bibitem{trinajsti} N. Trinajsti\'{c}, Chemical Graph Theory, 2nd ed.,  CRC Press, 1992

\bibitem{wiener1947} H. Wiener, Structural determination of paraffin boiling points, J. Am. Chem. Soc. 69 (1947) 17-20.
\end{thebibliography}
\end{document}